\documentclass{amsart}

\usepackage{lipsum}
\usepackage{amssymb}
\usepackage{graphicx}
\usepackage{epstopdf}
\usepackage{dsfont}
\usepackage{hyperref}
\usepackage{color}
\usepackage{mathrsfs}
\usepackage{subfigure}
\usepackage{bm}
\usepackage{algorithm}
\usepackage{algorithmicx}
\usepackage{tikz}
\usetikzlibrary{patterns}

\newcommand{\xkh}[1]{\left(#1\right)}
\newcommand{\dkh}[1]{\left\{#1\right\}}
\newcommand{\zkh}[1]{\left[#1\right]}

\newcommand{\nj}[1]{\left \langle {#1} \right \rangle}

\newcommand{\norm}[1]{\|{#1}\|_2}
\newcommand{\normf}[1]{\|{#1}\|_F}
\newcommand{\normone}[1]{\|{#1}\|_1}
\newcommand{\norms}[1]{\|{#1}\|}
\newcommand{\abs}[1]{\lvert#1\rvert}
\newcommand{\Abs}[1]{\left\lvert#1\right\rvert}

\newcommand{\PC}[2]{{\mathbb P} \left({#1} \: | \: {#2} \right)}

\newcommand{\A}{{\mathcal A}}

\newcommand{\E}{{\mathbb E}}

\newcommand{\PP}{{\mathbb P}}

\newcommand{\1}{{\mathds 1}}

\newcommand{\R}{{\mathbb R}}

\newcommand{\T}{\top}
\newcommand{\C}{{\mathbb C}}
\newcommand{\Cn}{{\mathbb C}^n}

\newcommand{\vx}{{\bm x}}
\newcommand{\vy}{{\bm y}}
\newcommand{\vu}{{\bm u}}

\newcommand{\vw}{{\bm w}}
\newcommand{\vz}{{\bm z}}
\newcommand{\vf}{{\bm f}}
\newcommand{\vd}{{\bm d}}

\newcommand{\md}{{\bm D}}
\newcommand{\mx}{{\bm X}}
\newcommand{\mQ}{{\bm Q}}

\newcommand{\vb}{{\bm b}}

\newcommand{\va}{{\bm a}}

\newcommand{\Z}{{\bm Z}}
\newcommand{\Y}{{\bm Y}}

\newcommand{\rank}{{\rm rank}}
\newcommand{\tr}{{\rm tr}}

\newcommand{\diag}{{\rm diag}}

\newtheorem{definition}{Definition}[section]

\newtheorem{theorem}[definition]{Theorem}
\newtheorem{lemma}[definition]{Lemma}
\newtheorem{assumption}[definition]{Assumption}
\newtheorem{remark}[definition]{Remark}

\newtheorem{proposition}[definition]{Proposition}

\numberwithin{equation}{section}

\begin{document}

\author{Meng Huang}
\address{School of Mathematical Sciences, Beihang University, Beijing, 100191, China} 
\email{menghuang@buaa.edu.cn}

\author{Jinming~Wen}
\address{Department of Mathematics, Jilin University, Changchun, China}
\email{jinming.wen@mail.mcgill.ca}

\author{Ran~Zhang}
\address{Department of Mathematics, Jilin University, Changchun, China}
\email{zhangran@jlu.edu.cn}


\date{}



\title[Recovery bound of PhaseLift from CDP]{Recovery Performance of PhaseLift for Phase Retrieval from Coded Diffraction Patterns}

\maketitle

\begin{abstract}
The PhaseLift algorithm is an  effective convex method for solving the phase retrieval problem from Fourier measurements with coded diffraction patterns (CDP). While exact reconstruction guarantees are well-established in the noiseless case, the stability of recovery under noise remains less well understood. In particular, when the measurements are corrupted by an additive noise vector $\vw \in \R^m$, existing recovery bounds scale on the order of $\norm{\vw}$, which is conjectured to be suboptimal. More recently, Soltanolkotabi conjectured that the optimal PhaseLift recovery bound should scale with the average noise magnitude, that is, on the order of $\norm{\vw}/\sqrt m$. However, establishing this theoretically  is considerably more challenging and has remained an open problem. In this paper, we focus on this conjecture and  prove that under adversarial noise, the recovery error of PhaseLift is bounded by $O\xkh{ \sqrt{\frac{\norm{\vw}\log n }{\sqrt m}}}\norm{\vx_0}$. Here, $\vx_0 \in \C^n$ is the signals we aim to recover. Moreover, for mean-zero sub-Gaussian noise vector $\vw \in \R^m$, a upper error bound and its corresponding minimax lower bound are also provided. Our results represent a significant step toward Soltanolkotabi's conjecture, offering new insights into the stability of PhaseLift under noisy CDP measurements.
\end{abstract}

%

\section{Introduction}
\subsection{Problem setup} \label{sec:probset}
Let $\vx_0 \in \Cn$ be an arbitrary unknown vector. The Fourier phase retrieval problem aims to recover $\vx_0$ from the modulus of its Fourier transform:
\[
y_k = \Abs{\sum_{t=0}^{n-1} x_0(t) e^{-\frac{2\pi i k t}{N}}}^2, \quad k = 1, \ldots, N-1,
\]
where $N \ge 2n-1$. This problem is equivalent to recovering $\vx_0$ from its auto-correlation, which is generally ill-posed \cite{beinert2018,bendory,edidin2019,huangJoC}. In fact, for a given signal dimension $n$, besides the trivial ambiguities caused by shift, conjugate reflection, and rotation, there can be up to $2^{n-2}$ nontrivial solutions. To address this issue, a popular approach to guarantee the uniqueness of recovery is to utilize multiple masks that introduce redundancy into the acquired data. This setup is called the {\em coded diffraction pattern} (CDP) model \cite{CDPLi,CDPGross}, and the noisy measurements we obtained are
\begin{equation*}
y_{k,l} = \Abs{\sum_{t=0}^{n-1} x_0(t) \bar{d}_l(t) e^{-\frac{2\pi i k t}{n}}}^2 + w_{k,l}, \quad 1 \le k \le n, \; 1 \le l \le L,
\end{equation*}
where $\vd_l \in \C^n$ are $L$ known masks and $w_{k,l}$ are noises. Here, $\bar{d}_l(t)$ denotes the conjugate of $d_l(t)$ for any $t=1,\ldots,m$.
 In matrix notation, this model can be written as
\begin{equation} \label{eq:measu}
y_{k,l} = \Abs{\nj{\md_l \vf_k, \vx_0}}^2 + w_{k,l}, \quad 1 \le k \le n, \; 1 \le l \le L,
\end{equation}
where $\vf_k^* \in \Cn$ is the $k$th row of the discrete Fourier transform (DFT) matrix, and $\md_l = \diag(\vd_l) \in \C^{n \times n}$ is a diagonal matrix containing the entries of the $l$th mask. 

There are several experimental techniques to generate masked Fourier measurements in optical setups, such as inserting a mask or a phase plate after the object \cite{liu2008phase,Loewen}.  Fourier phase retrieval has attracted significant attention across diverse fields of physical science and engineering, including X-ray crystallography \cite{Xcry1,Xcry2}, astronomy \cite{astronomy}, diffraction imaging \cite{shechtman2015phase,chai2010array}, microscopy \cite{miao2008extending}, as well as optics and acoustics \cite{walther1963question, balan2010signal, balan2006signal}, where the detector records only the diffracted intensity while the phase information is lost.

To recover $\vx_0$ from \eqref{eq:measu}, a commonly used and effective convex program, known as {\em PhaseLift}, was proposed \cite{CDPLi,CDPGross}. Specifically, let $\mx_0 = \vx_0 \vx_0^* \in \C^{n \times n}$. Then the measurements \eqref{eq:measu} can be rewritten as
\begin{equation} \label{eq:measurmatx}
y_{k,l} = \nj{\md_l \vf_k \vf_k^* \md_l^*, \mx_0} + w_{k,l}, \quad 1 \le k \le n, \; 1 \le l \le L,
\end{equation}
where,  $\nj{\cdot,\cdot}$ denotes the Hilbert–Schmidt inner product which is defined as $\nj{\Y, \Z} = \tr(\Y \Z)$ for any  two Hermitian matrices $\Y, \Z$.
Then the  phase retrieval problem thus becomes recovering a rank-one matrix $\mx_0$ from $\vy = \{y_{k,l}\}_{1 \le k \le n, 1 \le l \le L}$.  
 For convenience, we define a linear map $\mathcal{A} : \mathcal{H}^{n \times n} \to \R^{nL}$ as
\begin{equation} \label{def:A}
\mathcal{A}(\mx) = \bigl\{\nj{\md_l \vf_k \vf_k^* \md_l^*, \mx}\bigr\}_{1 \le k \le n, 1 \le l \le L} = \bigl\{\tr\bigl(\md_l \vf_k \vf_k^* \md_l^* \mx\bigr)\bigr\}_{1 \le k \le n, 1 \le l \le L},
\end{equation}
where $\mathcal{H}^{n \times n}$ denotes the set of all Hermitian matrices. Let $\vw = \{w_{k,l}\}_{1 \le k \le n, 1 \le l \le L}$. Assuming $\norm{\vw} \le \tau$, the following PhaseLift convex program \cite{CDPLi,CDPGross} can be used to estimate  $\mx_0$ from \eqref{eq:measurmatx}:
\begin{equation} \label{mod:con}
\begin{array}{ll}
\min_{\mx \succeq \bm{0}} & \tr(\mx) \\[0.5em]
\text{subject to} & \norm{\mathcal{A}(\mx) - \vy} \le \tau.
\end{array}
\end{equation}
Here, $\mx \succeq \bm{0}$ denotes $\mx$ is a positive semidefinite matrix.

We are interested in the following question:

{\em What is the optimal error bound of the program \eqref{mod:con}, and can it be rigorously established?}

\subsection{Motivation} \label{sec:motiv}
Recovery guarantees for the PhaseLift algorithm  from masked Fourier measurements were first established in \cite{CDPLi}. They showed that in the noiseless case, i.e., $\vw = \bm{0}$, if the masks are chosen at random and the number of masks satisfies $L \ge O(\log^4 n)$, then with high probability the solution to \eqref{mod:con} exactly recovers the target signal $\mx_0:=\vx_0\vx_0^*$. Later, Gross et al. improved this result to require only $O(\log^2 n)$ masks \cite{CDPGross}. 

In the presence of noise, the first stability result was given by Soltanolkotabi in \cite{Mahdithesis}. Specifically, they considered the convex program
\begin{equation} \label{mod:mahdi}
\min_{\mx \succeq \bm{0}} \quad \norm{\mathcal{A}(\mx) - \vy},
\end{equation}
and showed that when the masks are chosen randomly and $L \ge O(\log^4 n)$, with high probability the solution $\hat{\mx}$ to \eqref{mod:mahdi} with $\vy = \mathcal{A}(\mx_0) + \vw$ satisfies
\begin{equation} \label{bound:mahdi}
\normf{\hat{\mx} - \mx_0} \le C \norm{\vw},
\end{equation}
where $C > 0$ is a fixed numerical constant.

 In \cite{jaganathan2015phase}, the authors investigated two specific masks and assumed that each measurement is contaminated by bounded noise, i.e., $w_{k,l} \le \varepsilon$. They suggested estimating the signal by solving the following convex program:
\begin{equation} \label{mod:Eldar}
\min_{\mx \succeq \bm{0}} \quad \tr(\mx) \quad \text{subject to} \quad \norms{\mathcal{A}(\mx) - \vy}_{\infty} \le \varepsilon.
\end{equation}
They proved that if the signal $\vx_0 \in \Cn$ satisfies $\normone{\vx_0} \le \beta$ and $|x_0[0]| \ge \gamma $ for some constants $\beta,\gamma >0$, then the solution $\hat{\mx}$ to \eqref{mod:Eldar} obeys
\begin{equation} \label{eq:errEldar}
\normf{\hat{\mx} - \mx_0} \le C(\beta, \gamma) \varepsilon,
\end{equation}
where $C(\beta, \gamma)$ is a numerical constant. Although the error bound \eqref{eq:errEldar} is much smaller than the bound \eqref{bound:mahdi}, however, it requires additional conditions on $\vx_0$ and   assumptions on the noise vector.

More recently, \cite{li2021phase} revisited the problem and showed that if $L \ge O(\log^2 n)$ and $\norm{\vw} \le \tau$, then with high probability the solution $\hat{\mx}$ to \eqref{mod:con} satisfies
\begin{equation} \label{mod:li}
\normf{\hat{\mx} - \mx_0} \le C \sqrt{\log n} \cdot \tau.
\end{equation}
It is worth noting that  for Gaussian random measurements, the PhaseLift recovery bound  is known to be \cite{Phaseliftn}
\begin{equation} \label{eq:conjecture}
\normf{\hat{\mx} - \mx_0} \le C \frac{\norm{\vw}}{\sqrt{nL}}.
\end{equation}
Comparing this with the results \eqref{bound:mahdi} and \eqref{mod:li} for the CDP model, there is a significant gap. As stated on page 173 of \cite{Mahdithesis}, Soltanolkotabi {\em conjectures} that the optimal PhaseLift bound for the CDP model should be the Gaussian-like bound \eqref{eq:conjecture}, and further note that ``establishing this conjecture is a very interesting open problem and is significantly challenging."

Performance bounds of the PhaseLift algorithm have also been studied in blind deconvolution \cite{Ahmed,jung2017blind,Ling} and matrix completion \cite{Klopp,Koltchinskii}, where the noise bounds exhibit seemingly suboptimal dimension factors. Recently, Krahmer and Stöger \cite{krahmer2021convex} proved that the dimension factors in the noise bounds cannot be removed when the noise level is small. Since phase retrieval from the CDP model shares many similarities with blind deconvolution, it is natural to ask whether the noise bound \eqref{eq:conjecture} can be achieved here.
In this paper, we aim to establish a nearly sharp error bound for \eqref{mod:con} that positively resolves Soltanolkotabi's conjecture up to a logarithmic factor.

\subsection{Related work}
The phase retrieval problem, which aims to recover $\vx_0 \in \C^n$ from phaseless amplitude measurements
\[
y_k = \Abs{\nj{\va_k, \vx_0}}^2, \quad k = 1, \ldots, m,
\]
has recently been the subject of intensive research. Here, $\va_1, \ldots, \va_m \in \Cn$ are known measurement vectors. It has been shown theoretically that $m \ge 4n - 4$ generic measurement vectors suffice to recover $\vx_0$ up to a global phase in the complex case \cite{conca2015algebraic}, and $m \ge 2n - 1$ are sufficient in the real case \cite{balan2006signal}. 
Many algorithms with provable performance guarantees have been developed to solve the phase retrieval problem. These methods can be broadly categorized into two groups: those based on random measurement vectors for theoretical analysis, and those based on structured measurements motivated by practical applications.

\subsubsection{Phase retrieval based on random measurements}
For measurement vectors drawn independently at random from a Gaussian distribution, it has been shown that when the number of measurements satisfies $m \ge O(n)$, many efficient algorithms can guarantee stable recovery with high probability. 
One line of research relies on a “matrix-lifting” technique, which lifts the phase retrieval problem into a low-rank matrix recovery problem, solved via convex relaxation. Such methods include PhaseLift \cite{phaselift,Phaseliftn}, PhaseCut \cite{Waldspurger2015}, among others. In particular, Candès et al. considered the PhaseLift algorithm
\begin{equation} \label{mod:con1}
\begin{array}{ll}
\min_{\mx \succeq \bm{0}} & \tr(\mx) \\[0.5em]
\text{subject to} & \norm{\mathcal{A}(\mx) - \vy} \le \tau.
\end{array}
\end{equation}
Here, with a slight abuse of notation, we denote $\mathcal{A}: \mathcal{H}^{n \times n} \to \R^m$ by $\mathcal{A}(\mx) = \{ \va_k^* \mx \va_k \}_{k=1}^m$. They showed in \cite{phaselift} that when $\va_k \in \Cn$, $k=1,\ldots,m$ are independent and identically distributed (i.i.d.) complex Gaussian random vectors and $m \ge O(n \log n)$, then with high probability the solution $\hat{\mx}$ to \eqref{mod:con1} with $\vy = \mathcal{A}(\mx_0) + \vw$ and $\norm{\vw} \le \tau$ satisfies
\begin{equation} \label{bound:mahdi1}
\normf{\hat{\mx} - \mx_0} \le C \tau,
\end{equation}
for some numerical constant $C > 0$. Later, in \cite{Phaseliftn}, Candès and Li proposed the following empirical loss minimization to estimate $\mx_0$:
\begin{equation} \label{mo:phaselift}
\min_{\mx \in \C^{d \times d}} \quad \normone{\mathcal{A}(\mx) - \vy} \quad \text{subject to} \quad \mx \succeq 0.
\end{equation}
They proved that when $\va_k \in \Cn$, $k=1,\ldots,m$ are i.i.d. complex Gaussian random vectors and $m \ge O(n)$, with high probability the solution $\hat{\mx}$ to \eqref{mo:phaselift} satisfies
\[
\normf{\hat{\mx} - \mx_0} \le C_0 \frac{\normone{\vw}}{m},
\]
for some numerical constant $C_0 > 0$. This substantially improves upon \eqref{bound:mahdi1}.

Due to the computational inefficiency of PhaseLift in large-scale problems, another line of research focuses on optimizing a non-convex loss function in the natural parameter space, achieving significantly improved computational performance via a technique known as {\em spectral initialization}. Notable examples include \cite{AltMin, WF, TAF, TWF, RWF, turstregion, waldspurger2018phase, cai2021solving, huangkaz,zhaoYB,caiIP,caiHIP}, among others. Regarding the stability of empirical risk minimization, a sharp recovery bound for a commonly used non-convex estimator under adversarial noise was provided in \cite{huangPerfm}, and nearly minimax error bounds under sub-Gaussian noise were established in \cite{Eldarstability, lecuMini}.
For comprehensive recent developments in the theory, algorithms, and applications of phase retrieval, we refer the readers to survey papers \cite{jaganathan2016phase, shechtman2015phase}.

\subsubsection{Phase retrieval from structured measurements}

Although the theory and algorithms for phase retrieval based on Gaussian measurements were well-developed, they are not applicable to many practical scenarios. In reality, practical applications often involve structured measurements, with CDP being a commonly used setup, as described in Section \ref{sec:probset}. 
The first recovery guarantees from masked Fourier measurements were established for polarization-based recovery with highly specific masks. In the noiseless case, they showed \cite{Bandeira} that $O(\log n)$ masks suffice for unique recovery. Later, Candès et al. studied masks chosen at random and solved the phase retrieval problem using the PhaseLift algorithm, proving that $O(\log^4 n)$ random masks are sufficient for exact reconstruction of signals with high probability \cite{CDPLi}. This bound was subsequently improved to $O(\log^2 n)$ by Gross et al \cite{CDPGross}. Using the standard coupon collector's argument, they also showed that the lower bound for the number of Rademacher masks with random erasures required to guarantee uniqueness of recovery is $O(\log n)$. 
However, due to the much more structured and less random nature of coded diffraction patterns compared to Gaussian designs, achieving the theoretically optimal bound of $O(\log n)$ for CDP remains an open problem. In the presence of noise, stability results for PhaseLift have been provided in \cite{Mahdithesis, jaganathan2015phase}, as discussed in Section \ref{sec:motiv}.
To improve computational efficiency, several non-convex algorithms based on spectral initialization have been developed for phase retrieval with masks, demonstrating strong empirical performance alongside rigorous theoretical guarantees (see \cite{WF, li2022,lili} for further details).

\subsection{Our contributions}  \label{sec:contri}
As mentioned previously, PhaseLift algorithms are efficient convex methods for solving phase retrieval from CDP. Recovery guarantees for the noiseless case were provided in \cite{CDPLi,CDPGross}, while several recovery error bounds for the noisy case have been established in \cite{Mahdithesis,li2021phase}. However, it is widely believed that these error bounds are far from optimal. Establishing a sharp recovery bound for the CDP model is considerably more challenging and remains an open problem \cite[p. 173]{Mahdithesis}.  The goal of this paper is to address this problem. Throughout this paper, we adopt the commonly used assumption that the masks are chosen at random, as stated below. 

\begin{assumption}[\cite{CDPLi,CDPGross}] \label{assump:mask}
The entries of each mask $\vd_l \in \C^n$ are i.i.d. copies of a complex random variable $d$ which is symmetric and satisfies
\begin{equation*}
\E d = 0, \quad \E d^2 = 0, \quad \E |d|^4 = 2 \bigl(\E |d|^2 \bigr)^2, \quad |d| \le M,
\end{equation*}
for some fixed constant $M > 0$. For convenience, we denote $\nu = \E |d|^2$.
\end{assumption}
As shown in \cite{Fannjiang}, random masks are physically realizable and particularly suitable for phase retrieval. An example of a random variable $d$ satisfying Assumption \ref{assump:mask}  is given by $d = b_1 b_2$ with
\begin{equation} \label{eq:octanary}
b_1 \sim \left\{
\begin{array}{rl}
1 & \text{with prob. } 1/4 \\
-1 & \text{with prob. } 1/4 \\
-i & \text{with prob. } 1/4 \\
i & \text{with prob. } 1/4
\end{array}
\right., \quad \text{and} \quad
b_2 \sim \left\{
\begin{array}{ll}
1/\sqrt{2} & \text{with prob. } 4/5 \\
\sqrt{3} & \text{with prob. } 1/5
\end{array}
\right.,
\end{equation}
which is referred to as octanary codes in \cite{CDPLi}. Our main results are stated as follows.

\begin{theorem} \label{mainresult}
Let $\vx_0 \in \Cn$ and $\omega \ge 1$. Suppose that the masks $\dkh{\vd_l}_{l=1}^L$ satisfy Assumption \ref{assump:mask}, and the number of masks $L$ satisfies $L \ge C_0 \log^2 n$ for some constant $C_0 > 0$ depending only on $M$ and $\nu$. For any noise vector $\vw \in \R^m$ with $\norm{\vw} \le \tau$, with probability at least $1 - e^{-\omega}$, the solution $\hat{\mx}$ to \eqref{mod:con} with $\vy = \mathcal{A}(\mx_0) + \vw$ satisfies
\begin{equation} \label{eq:newbo}
\normf{\hat{\mx} - \mx_0} \le \norm{\vx_0} \cdot \min\dkh{2\norm{\vx_0}, C\sqrt{\frac{ \tau \log n}{\sqrt{nL}}}}.
\end{equation}
Here,  $\mx_0 = \vx_0 \vx_0^*$, and $C > 0$ is a constant only depends on $M, \nu$.
\end{theorem}

\begin{remark}
In the noise regime where $\tau \gg \frac{\log n}{\sqrt{nL}} \norm{\vx_0}^2$, the  bound \eqref{eq:newbo} significantly improves over the best known result \eqref{bound:mahdi}. Furthermore, in the noise level regime where $\tau = c \norm{\vx_0}^2 \sqrt{nL}$ for some sufficiently small constant $c>0$, which is a practically relevant regime, the  bound \eqref{eq:newbo} becomes 
\[
\normf{\hat{\mx} - \mx_0} \le c' \sqrt{\log n}\norm{\vx_0}^2, 
\]
where $c'>0$ is a sufficiently small constant. This result nearly matches the bound conjectured by Soltanolkotabi  in  \eqref{eq:conjecture}. 
\end{remark}

\begin{remark} \label{re:advn}
Under the assumptions of Theorem \ref{lowbound}, using the same approach as in \cite{phaselift}, we can obtain an estimate $\hat{\vx} \in \C^n$ by finding the leading eigenvector corresponding to the largest eigenvalue of $\hat{\mx}$ such that
\[
\mbox{dist}(\hat\vx,\vx_0) \le C \min \left\{\norm{\vx_0},  \sqrt{\frac{ \tau \log n}{\sqrt{nL}}} \right\}.
\]
Here, the distance is defined as $\mbox{dist}(\hat\vx,\vx_0):=\min_{\phi \in [0, 2\pi)} \norm{\hat{\vx} - e^{i \phi} \vx_0}$, since we can recover $\vx_0$ up to a global phase.
\end{remark}

The following theorem shows that the error bound $\norm{\vw}/\sqrt{nL}$ is sharp up to a logarithmic factor.

\begin{theorem} \label{lowbound}
Let $\vx_0 \in \Cn$. Suppose that the masks $\dkh{\vd_l}_{l=1}^L$ satisfy Assumption \ref{assump:mask} and the number of masks $L$ satisfies $L \ge C_0 \log n$ for some constant $C_0 > 0$ depending only on $M$ and $\nu$. Then, with probability at least $1 - 4Ln^{-10}$, there exists a noise vector $\vw \in \R^{nL}$ and parameter $\tau \le \sqrt{nL \log n} \norm{\vx_0}^2$ such that the solution $\hat{\mx}$ to \eqref{mod:con} with $\vy = \mathcal{A}(\mx_0) + \vw$ and $\norm{\vw} \le \tau$ satisfies
\[
\normf{\hat{\mx} - \mx_0} \ge C_1 \frac{\tau}{\sqrt{nL \log n}}.
\]
Here, $C_1 > 0$ is a constant depending only on $M$, $\nu$.
\end{theorem}

For some practical applications, the measurements are contaminated by Gaussian noise. In this case, a tighter error bound can be established. For Gaussian noise, we consider the following convex program:
\begin{equation} \label{model:gau}
\min_{\mx \succeq \bm{0}} \quad \norm{\mathcal{A}(\mx) - \vy} \quad \text{subject to} \quad \tr(\mx) \le R,
\end{equation}
where $R>0$ is a parameter which specifies the desired rank level of the solution. The following theorem presents the estimation performance of the program \eqref{model:gau}.

\begin{theorem} \label{UpperGaussian}
Let $\vx_0 \in \Cn$ and $\omega \ge 1$. Suppose that the masks $\dkh{\vd_l}_{l=1}^L$ satisfy Assumption \ref{assump:mask}  and the number of masks $L$ satisfies $L \ge C_0 \,\omega \log^2 n$ for some constant $C_0 > 0$ depending only on $M$ and $\nu$. Assume that the noise $\vw \in \R^m$ is a mean-zero sub-Gaussian random vector with $\norms{\vw}_{\psi_2} \le \sigma$ for some constant $\sigma > 0$. Then, with probability at least $1 - e^{-\omega}$, the solution $\hat{\mx}$ to \eqref{model:gau} with $\vy = \mathcal{A}(\mx_0) + \vw$ and $R = \norm{\vx_0}^2$ satisfies
\[
\normf{\hat{\mx} - \mx_0} \le C_2 \norm{\vx_0} \sqrt{\frac{ \sigma  \log^2 n}{\sqrt{L}}}.
\]
Here, $C_2 > 0$ is a constant depending on $\nu$, $M$.
\end{theorem}

\begin{remark} \label{re:1.5}
For Gaussian noise, the error bound given in Theorem \ref{UpperGaussian} is tighter than that in Theorem \ref{mainresult}. To see this, observe that when $\vw \sim \mathcal{N}(0, \sigma^2I_{m})$, it holds that $\norm{\vw} = O(\sigma \sqrt{m})$ with high probability. Here, $m:=nL$.
 Therefore, the error bound in Theorem \ref{mainresult} is  $\normf{\hat{\mx} - \mx_0} \le C \sqrt{\sigma \log n}\norm{\vx_0}$ for a constant $C>0$. However, the error bound in Theorem \ref{UpperGaussian} approaches zero as $L \to \infty$.  
\end{remark}

\begin{remark}
Similar to Remark \ref{re:advn},  under the assumptions of Theorem \ref{UpperGaussian},  an estimate $\hat{\vx} \in \C^n$ can be  constructed from the solution   $\hat{\mx}$ to \eqref{model:gau}  such that
\begin{equation} \label{eq:upgaumini}
\mbox{dist}(\hat\vx,\vx_0)  \le C_2 \min \left\{\norm{\vx_0}, \sqrt{\frac{ \sigma \log^2 n}{\sqrt{L}}} \right\}.
\end{equation}
\end{remark}

We next give a minimax error bound under the standard Gaussian noise.  In  the following theorem, we focus on real-valued signals $\vx_0 \in \R^n$, and the distance between a estimator $\hat \vx \in \R^n$ and $\vx_0$ is defined as $\mbox{dist}(\hat\vx,\vx_0)=\min\dkh{\hat\vx-\vx_0, \hat\vx+\vx_0}$.

\begin{theorem} \label{th:minimax}
Suppose that the masks $\dkh{\vd_l}_{l=1}^L$ satisfy Assumption \ref{assump:mask},   and the number of masks $L \le  C_0 \log^k n$ for some fixed integer $k\ge 1$ and constant $C_0 > 0$ independent of $n$, with $n$ is sufficiently large. Assume that the noises $\dkh{w_{k,l}}$ are independent mean-zero Gaussian random variables with variance $\sigma^2$, i.e.,  $\vw_{k,l} \sim \mathcal{N}(0,\sigma^2)$. Then, with probability approaching $1$, the minimax risk under the Gaussian model \eqref{eq:measu} obeys
\[
\inf_{\hat\vx \in \R^n} \sup_{\vx_0 \in \R^n \atop \norm{\vx_0} \ge \sigma} \E\zkh{ {\rm{dist}}(\hat\vx,\vx_0)  |  \dkh{\vd_l}_{l=1}^L} \ge \frac{c_0 \sigma}{\sqrt{L}  \log n  \norm{\vx_0}},
\]
where the infimum is over all estimators $\hat \vx$, and $c_0>0$ is a constant depending only on $M$. 

\end{theorem}

\begin{remark}
Theorem \ref{th:minimax} shows that  for Gaussian noise, it may be  possible to reduce the error bound to
\[
\mbox{dist}(\hat\vx,\vx_0)  \le \frac{C\sigma}{\sqrt{L} \norm{\vx_0}}
\]
for a constant $C>0$. How to achieve this improvement is an interesting problem and a direction for future work.
\end{remark}

\subsection{Notations}
Throughout this paper, we denote by $\mathcal{H}^{n \times n}$ the set of all Hermitian matrices. The trace of a matrix is denoted by $\tr(\cdot)$. For any two matrices $\Y, \Z \in \mathcal{H}^{n \times n}$, the Hilbert–Schmidt inner product is defined as $\nj{\Y, \Z} = \tr(\Y \Z)$, and we write $\Y \succeq \Z$ if and only if $\Y - \Z$ is positive semidefinite. 
We denote by $\norm{\cdot}$, $\normf{\cdot}$, and $\norms{\cdot}_*$ the operator norm, Frobenius norm, and nuclear norm of a matrix, respectively. For a matrix-valued operator $\mathcal{A}$ acting on $\mathcal{H}^{n \times n}$, the operator norm is defined as
\[
\norm{\mathcal{A}} = \sup_{\Z \in \mathcal{H}^{n \times n}} \frac{|\tr(\Z \mathcal{A}(\Z))|}{\normf{\Z}}.
\]
For a vector $\vz \in \Cn$, we use $\vz^\T$ and $\vz^*$ to denote the transpose and the conjugate transpose of $\vz$, respectively. For a real number $a \in \R$, we use $\lceil a \rceil$ to denote the smallest integer that is greater than or equal to $a$.

\subsection{Organization}
The paper is organized as follows. In Section \ref{sec:2}, we introduce some notations and definitions that will be used throughout the paper. In particular, the concepts of robust injectivity and approximate dual certificates play a key role in proving the main results. 
In Section  \ref{sec:3}, we first show how to construct an exact dual certificate from an approximate dual certificate, followed by the proof of the main results. Section  \ref{sec:4} presents numerical experiments that validate the optimality of our theoretical findings. 
A brief discussion is provided in Section  \ref{sec:5}. The appendix contains the proofs of technical lemmas.

\section{Preliminaries} \label{sec:2}
The aim of this section is to introduce some technical lemmas that will be used throughout the paper. 
Let $\vx_0 \in \Cn$ be the target signal we wish to recover. The measurements we obtain are
\[
\vy = \mathcal{A}(\vx_0 \vx_0^*) + \vw,
\]
where $\mathcal{A}$ is defined in \eqref{def:A} and $\vw$ is a noise vector.  The corresponding adjoint operator $\A^*$ of the linear map $\A$ is defined as
\begin{equation} \label{eq:adj}
\A^*: \R^{nL} \to \mathcal H^{n\times n} \qquad \vb \to  \A^*(\vb)=\sum_{l=1}^L \sum_{k=1}^n  b_{k,l} \md_l\vf_k \vf_k^* \md_l^*.
\end{equation}
Without loss of generality, we assume $\norm{\vx_0} = 1$.
The tangent space to the manifold of all rank-$1$ Hermitian matrices at $\mx_0 := \vx_0 \vx_0^*$ is given by
\begin{equation} \label{eq:Tangent}
T = \{ \vx_0 \vz^* + \vz \vx_0^* : \vz \in \Cn \} \subset \mathcal{H}^{n \times n},
\end{equation}
which is a subspace of Hermitian matrices $\mathcal{H}^{n \times n}$.  Let $T^\perp$ denote its orthogonal complement under the Frobenius inner product. Then any Hermitian matrix $\Z \in \mathcal{H}^{n \times n}$ can be decomposed as
\[
\Z = \mathcal{P}_T \Z + \mathcal{P}_{T^\perp} \Z := \Z_T + \Z_{T^\perp}.
\]
Throughout the paper, we use $\mathcal{P}_T \Z$ or $ \Z_T$ to denote the orthogonal projection of $\Z$ onto the tangent space $T$.

\subsection{Robust injectivity and uniform upper bound}
We present several intermediate results that will be used in the paper. 
The following lemma shows that the linear map $\mathcal{A}$ is robustly injective on the tangent space $T$ at $\mx_0$ with high probability, provided that the number of masks satisfies $L \ge O(\log n)$.

\begin{lemma}\cite[Proposition 8]{CDPGross}  \label{le:injective}
Assume that the number of masks $L \ge C_0  \log n$ for some constant $C_0$ depending only on $\nu, M$. Then with probability at least $1-\frac1n$, it holds
\[
\frac{1}{\nu \sqrt{nL}} \norm{\A(\Z)} > \frac12 \normf{\Z}
\]
for all matrices $\Z \in T$. Here, $\A$ is given in \eqref{def:A} and $\nu, M$ are as in Assumption \ref{assump:mask}.
\end{lemma}

A uniform upper bound on $\norm{\A(\Z)} $ for all matrices $\Z \in \mathcal H^{n\times n}$ is given below.

\begin{lemma}\cite[Lemma 10]{CDPGross} \label{le:opnorm}
Let $\A$ be defined in \eqref{def:A}. Then  it holds
\[
\frac{1}{\sqrt{nL}} \norm{\A(\Z)} \le M^2 \sqrt{n} \normf{\Z}
\]
for all matrices $\Z \in \mathcal H^{n\times n}$. 
\end{lemma}

\subsection{Approximate  dual certificates}
In \cite{CDPGross}, the so-called approximate dual certificate was constructed via the widely used golfing scheme first introduced by Gross \cite{Grossgolf}. The dual certificate is employed to guarantee exact reconstruction of the convex program \eqref{mod:con} with the aid of robust injectivity in the absence of noise. We present the result here, including an additional derivation of the norm upper bound for the dual certificate.

\begin{lemma}[Approximate dual certificate] \label{le:appdual}
Let $\omega \ge 1$. If the number of masks satisfies $L\ge C_0 \omega \log^2n$, then with probability at least $1-5/6 e^{-\omega}$, there exists an approximate dual certificate pair $(\Y',\lambda')$ such that $\Y'=\A^*(\lambda')$ and 
\begin{equation} \label{eq:appdual1}
\normf{\mathcal P_T \Y'- \vx_0\vx_0^*} \le \frac{\nu}{8 M^2 \sqrt{n}}  \quad \mbox{and} \quad  \norm{\mathcal P_{T^\perp} \Y'} \le \frac12.
\end{equation}
Furthermore, it holds
\begin{equation} \label{eq:appdual2}
\norm{\lambda'} \le \frac{c' \log n}{\sqrt{nL}}.
\end{equation}
Here, $C_0, c'>0$ are constants only depend on $M,\nu$, and $\A^*$ is the adjoint operator as in \eqref{eq:adj}.
\end{lemma}

\begin{proof}
See Appendix A.
\end{proof}

\section{Proof of Main Result} \label{sec:3}
In this section, we present the proofs of our main results. We begin by showing that an exact dual certificate can be constructed from the approximate dual certificate using the proposition from \cite{fuchs2022proof}, as stated below.

\begin{lemma} \label{le:exactdual}
Let $\omega \ge 1$. Assume that the number of masks obeys $L \ge C_0 \omega \log^2 n$. Then with probability at least $1- 5/6 e^{-\omega}$, there exists an exact dual certificate pair $(\Y,\lambda)$ such that $\Y=\A^*(\lambda)$ and 
\[
\mathcal P_T \Y=\vx_0 \vx_0^*, \quad \norm{\mathcal P_{T^\perp} \Y} \le \frac34, \quad \mbox{and} \quad \norm{\lambda} \le \frac{c \log n}{\sqrt{nL}}.
\]
Here, $C_0, c>0$ are constants only depend on $M,\nu$, and $\A^*$ is the adjoint operator as in \eqref{eq:adj}.
\end{lemma}
\begin{proof}
The proof is adapted from \cite{fuchs2022proof}. According to Lemma \ref{le:appdual}, when  $L\ge C_0 \omega \log^2n$, with probability at least $1-5/6 e^{-\omega}$, there exists an approximate dual certificate pair $(\Y',\lambda')$ such that $\Y'=\A^*(\lambda')$ and 
\begin{equation} \label{eq:appdual11}
\normf{\mathcal P_T \Y'- \vx_0\vx_0^*} \le \frac{\nu}{8 M^2 \sqrt{n}}  \quad \norm{\mathcal P_{T^\perp} \Y'} \le \frac12,\quad \norm{\lambda'} \le \frac{c' \log n}{\sqrt{nL}}.
\end{equation}
Here, $C_0, c'>0$ are constants only depends on $M,\nu$. Considering the  operate $\mathcal P_T \A^* \A \mathcal P_T:T \to  T$, it follows from  Lemma \ref{le:injective} that it is invertible and
\[
\norm{\xkh{\mathcal P_T \A^* \A \mathcal P_T}^{-1}} \le \frac2{\nu \sqrt{nL}}.
\]
Define 
\[
\lambda= \lambda' + \A \mathcal P_T \xkh{\mathcal P_T \A^* \A \mathcal P_T}^{-1} \xkh{\vx_0\vx_0^* - \mathcal P_T \Y' } \quad \mbox{and}\quad \Y=\A^*(\lambda).
\]
Then one can verify that 
\[
 \mathcal P_T \Y=  \mathcal P_T \A^*( \lambda') + \mathcal P_T \A^*\A \mathcal P_T \xkh{\mathcal P_T \A^* \A \mathcal P_T}^{-1} \xkh{\vx_0\vx_0^* - \mathcal P_T \Y' }=  \vx_0\vx_0^*
\]
and
\begin{eqnarray*}
\norm{ \mathcal P_{T\perp} \Y} & =& \norm{ \mathcal P_{T\perp} \Y'+  \mathcal P_{T\perp}  \A^*\A \mathcal P_T \xkh{\mathcal P_T \A^* \A \mathcal P_T}^{-1} \xkh{\vx_0\vx_0^* - \mathcal P_T \Y' } } \\
&\le & \norm{\mathcal P_{T\perp} \Y'} + \norm{\A} \norm{ \A \mathcal P_T \xkh{\mathcal P_T \A^* \A \mathcal P_T}^{-1} } \normf{\vx_0\vx_0^* - \mathcal P_T \Y' } \\
&\le & \frac12+ M^2 n \sqrt{L} \cdot   \frac2{\nu \sqrt{nL}} \cdot \frac{\nu}{8 M^2 \sqrt{n}} = \frac34,
\end{eqnarray*}
where the second inequality comes from Lemma \ref{le:opnorm} that $ \norm{\A}  \le M^2 n \sqrt{L}$,  inequality \eqref{eq:appdual11}, and the  fact that
\begin{eqnarray*}
\norm{ \A \mathcal P_T \xkh{\mathcal P_T \A^* \A \mathcal P_T}^{-1} }^2  &=&  \sup_{\normf{\mx} \le 1}  \normf{ \A \mathcal P_T \xkh{\mathcal P_T \A^* \A \mathcal P_T}^{-1} \mx }^2 \\
& =&  \sup_{\normf{\mx} \le 1} \nj{ \A \mathcal P_T \xkh{\mathcal P_T \A^* \A \mathcal P_T}^{-1} \mx,  \A \mathcal P_T \xkh{\mathcal P_T \A^* \A \mathcal P_T}^{-1} \mx} \\
& =&  \sup_{\normf{\mx} \le 1} \nj{  \mx, \xkh{\mathcal P_T \A^* \A \mathcal P_T}^{-1} \mx} \\
& =& \norm{\xkh{\mathcal P_T \A^* \A \mathcal P_T}^{-1}} \le \frac2{\nu \sqrt{nL}}.
\end{eqnarray*}

Similarly, the norm of dual certificate $\lambda$ obeys
\begin{eqnarray*}
\norm{\lambda} & \le &  \norm{\lambda'}+ \norm{\A \mathcal P_T \xkh{\mathcal P_T \A^* \A \mathcal P_T}^{-1} \xkh{\vx_0\vx_0^* - \mathcal P_T \Y' }} \\
& \le &  \frac{c' \log n}{\sqrt{nL}}+\frac2{\nu \sqrt{nL}}\cdot \frac{\nu}{8 M^2 \sqrt{n}}  \le \frac{c \log n}{\sqrt{nL}}.
\end{eqnarray*}

\end{proof}

\subsection{Proof of Theorem \ref{mainresult}}
With the above lemma in hand, we are now prepared to establish an upper bound for the PhaseLift program \eqref{mod:con}. The proof is inspired by \cite{kostin}, where the norm $\norm{\mathcal{A}(\Z)}$ is bounded from below by its scalar product with the exact dual certificate.

\begin{proof}[Proof of Theorem \ref{mainresult}]
Note that $\hat \mx$ is the optimal solution to  \eqref{mod:con} and $\mx_0$ is a feasible point. Therefore, one has
\begin{equation} \label{eq:feasibleZ}
\hat \mx \succeq \bm 0, \quad \tr(\hat\mx) \le \tr(\mx_0), \quad \mbox{and} \quad \norm{\A(\hat\mx)-\vy} \le \tau.
\end{equation}
Denote $\hat\mx =\mx_0 +t \Z$ with $t>0$ and $\normf{\Z}=1$. Then we have
\[
\mx_0 +t \Z \succeq \bm 0, \quad \tr(\Z) \le 0,
\]
and 
\begin{equation} \label{eq:Azupp}
t\norm{\A(\Z)}=\norm{\A(\hat\mx)-\vy+\vw} \le \norm{\A(\hat\mx)-\vy} +\norm{\vw} \le 2\tau.
\end{equation}
We claim that $t\le 2$ and the matrix $\Z$ can be decomposed into 
\begin{equation} \label{eq:Zdecom}
\Z=-\beta \vx_0\vx_0^* + \vu \vx_0^* + \vx_0 \vu^* +  \Z_0,
\end{equation}
where $\beta \ge t/4$, $\nj{\vu,\vx_0}=0$,  and $\Z_0 \vx_0=0$ with $\norms{\Z_0}_* \le \beta$.  To prove the theorem, it then suffices to show 
\begin{equation} \label{eq:AZnorm}
\norm{\A(\Z)} \ge  \frac{\sqrt{nL}}{ 4 c \log n} \cdot \beta
\end{equation}
holds with probability at least $1- e^{-\omega}$, where $c$ is the constant as in Lemma \ref{le:exactdual}.
Indeed, if \eqref{eq:AZnorm} holds, then it follows from \eqref{eq:Azupp} that 
\[
\normf{\hat\mx -\mx_0} =\normf{t \Z} =t  \le \min\dkh{2, C \sqrt{\frac{ \tau \log n}{\sqrt{nL}}}},
\]
where $C=4\sqrt{2c}$ is a constant depends on $M,\nu$. This gives the conclusion.

We proceed to prove \eqref{eq:AZnorm}. From Lemma \ref{le:exactdual}, with probability at least $1- 5/6 e^{-\omega}$,  there exists an exact dual certificate pair $(\Y,\lambda)$ such that $\Y=\A^*(\lambda)$ and 
\begin{equation} \label{eq:AZ1}
\mathcal P_T \Y=\vx_0 \vx_0^*, \quad \norm{\mathcal P_{T^\perp} \Y} \le \frac34, \quad \norm{\lambda} \le \frac{c \log n}{\sqrt{nL}}.
\end{equation}
Applying Cauchy-Schwarz inequality, we have
\begin{eqnarray} \label{eq:AZ0}
\norm{\A(\Z)} \ge \frac1{\norm{\lambda}} \Abs{\nj{\A(\Z),\lambda}} &= & \frac1{\norm{\lambda}} \Abs{\nj{\Z,\A^*(\lambda)}} \notag \\
&= & \frac1{\norm{\lambda}} \Abs{\nj{\Z,\Y}}  \notag \\
&\ge&  \frac{\sqrt{nL}}{c\log n} \Abs{\nj{\Z,\Y}} .
\end{eqnarray}
Notice that 
\begin{equation} \label{eq:AZ2}
\nj{\Z,\Y}=\nj{\Z_T,\Y_T}+\nj{\Z_{\perp},\Y_{\perp}},
\end{equation}
where $\Z_T$ and $\Z_{\perp}$ denote the projection of $\Z$ onto tangent space $T$ and its orthogonal complement space $T^\perp$, respectively.
For the first term of the right hand side \eqref{eq:AZ2}, it follows from \eqref{eq:Zdecom} and \eqref{eq:AZ1} that 
\[
\nj{\Z_T,\Y_T}=\nj{-\beta \vx_0\vx_0^* + \vu \vx_0^* + \vx_0 \vu^*, \vx_0\vx_0^*}=-\beta.
\]
Similarly, the second term  of the right hand side \eqref{eq:AZ2} can be estimated as 
\[
\Abs{\nj{\Z_{\perp},\Y_{\perp}}} \le \norms{\Z_{\perp}}_{*} \norm{\Y_{\perp}} =\norms{\Z_0}_{*} \norm{\Y_{\perp}}  \le \frac34 \beta,
\]
where we use the fact that $\norms{\Z_0}_* \le \beta$ and $\norm{\mathcal P_{T^\perp} \Y} \le 3/4$ in the last inequality.
Combining the above three estimates, one has
\begin{equation} \label{eq:AZ3}
\Abs{\nj{\Z,\Y} } \ge \frac14 \beta.
\end{equation}
 Putting \eqref{eq:AZ3} into \eqref{eq:AZ0}, we obtain  \eqref{eq:AZnorm}.

It remains to prove $t \le 2$ and the claim \eqref{eq:Zdecom}. Note that the Hermitian matrix $\Z$ can be  decomposed into 
\begin{equation} \label{eq:mf0}
\Z=\mathcal P_T \Z+ \mathcal P_{T^\perp} \Z=-\beta \vx_0\vx_0^* + \vu \vx_0^* + \vx_0 \vu^* +  \Z_0,
\end{equation}
where $\beta \in \R$, $\nj{\vu,\vx_0}=0$,  and $\Z_0 \vx_0=0$. Here, $T$ is the tangent space to the manifold of all rank-$1$ Hermitian matrices at $\mx_0 := \vx_0 \vx_0^*$ as given in \eqref{eq:Tangent}.
Since $\mx_0 +t \Z \succeq \bm 0$, it gives $\Z_0= \mathcal P_{T^\perp} \xkh{\mx_0 +t \Z} \succeq \bm 0$.  
Combining with  the fact that $\tr(\Z) \le 0$, one has
\begin{equation*} 
\tr(\Z)=-\beta +\tr(\Z_0)=-\beta+\norms{\Z_0}_* \le 0.
\end{equation*}
Therefore, it holds
\begin{equation} \label{eq:Ztrace}
\norms{\Z_0}_* \le \beta.
\end{equation}
 Denote $\vu=\gamma \tilde{\vu}$ with $\norm{\tilde{\vu}}=1$ and $\gamma\ge 0$. We further set
 \[
 \Z_0=\tilde{\vu}\tilde{\vu}^* \Z_0 \tilde{\vu}\tilde{\vu}^*+\Z'_0,
 \]
where $\Z'_0 \in \C^{n\times n}$. One can easily check that $\vw^*  \Z'_0 \vw =0$ for all $\vw \in \mbox{span}\dkh{\vx_0, \tilde{\vu}}$. The matrix $\Z_0$ can be written as
\begin{equation} \label{eq:mf00}
\Z=-\beta \vx_0\vx_0^* + \gamma \tilde{\vu} \vx_0^* + \gamma \vx_0 \tilde{\vu}^* + \zeta \tilde{\vu} \tilde{\vu}^* + \Z'_0,
\end{equation}
where $ \zeta=\tilde{\vu}^* \Z_0 \tilde{\vu}\ge 0$.
Since $\vx_0\vx_0^* +t \Z \succeq \bm 0$, thus for any $\vw \in \mbox{span}\dkh{\vx_0, \tilde{\vu}}$, it holds
\[
\vw^* \xkh{\vx_0\vx_0^* +t \Z} \vw=\vw^*\xkh{(1- t\beta) \vx_0\vx_0^* + t \gamma \tilde{\vu} \vx_0^* + t \gamma \vx_0 \tilde{\vu}^* +t \zeta \tilde{\vu} \tilde{\vu}^*} \vw \ge 0.
\]
Note that $\vx_0$ is orthogonal to $\tilde{\vu}$, and $\norm{\vx_0}=\norm{\tilde{\vu}}=1$. It implies 
\[
\left( \begin{array}{cc} 1- t\beta & t\gamma\\ t \gamma & t\zeta \end{array} \right) \succeq \bm 0.
\]
This gives 
\[
\zeta \ge 0, \quad  1- t\beta \ge 0 \quad \mbox{and} \quad (1- t\beta) \cdot  t \zeta \ge t^2\gamma^2.
\]
Thus,
\[
\gamma^2 \le \frac{(1- t\beta)  \zeta}{t} \le \frac{\beta}{t},
\]
where the last inequality comes from the fact that $\zeta=\tilde{\vu}^* \Z_0 \tilde{\vu} \le \norm{\Z_0} \le \beta$ due to \eqref{eq:Ztrace}. Finally, according to \eqref{eq:mf0}, one has
\begin{equation} \label{eq:Znorm}
1=\normf{\Z}^2=\beta^2+2\gamma^2+\normf{\Z_0}^2 \le \beta^2+2\gamma^2+\norms{\Z_0}_*^2 \le 2\beta^2+2\cdot\frac{\beta}{t}.
\end{equation}
This implies
\[
\beta\ge t \xkh{\frac12-\beta^2}= \frac{1}{2}t- t\beta*\beta \ge  \frac{1}{2}t- \beta,
\]
which gives $\beta \ge t/4$. Here, the last inequality follows from the fact that $t\beta \le 1$. Moreover, combining $t\beta \le 1$ and  $\beta \ge t/4$ together gives $t\le 2$.
This complete the claim \eqref{eq:Zdecom}.
\end{proof}

\subsection{Proof of Theorem \ref{lowbound}}

\begin{proof}[Proof of Theorem \ref{lowbound}]
For any fixed $\vx_0 \in \Cn$, let the noise vector $\vw=-\A(\mx_0)$ and the parameter $\tau= \norm{\A(\mx_0)}$, where $\mx_0=\vx_0\vx_0^*$. Then we have 
\[
\vy=\A(\mx_0) +\vw=\bm 0.
\]
Therefore, the optimal solution to \eqref{mod:con} is $\hat \mx=\bm 0$ and $\normf{\hat\mx-\mx_0}=\normf{\mx_0}$.
We claim that, with probability at least   $1-4Ln^{-10}$, it holds $\tau= \norm{\A(\mx_0)} \le \sqrt{c_2 nL \log n}\normf{\mx_0}$ for a constant $c_2$ depending only on $M, \nu$. This immediately gives the conclusion that
\[
\normf{\hat\mx-\mx_0} =\normf{\mx_0} \ge \frac{C_1 \tau}{\sqrt{nL \log n}}.
\]
Here, $C_1=1/\sqrt{c_2}$ is a constant depending only on $M, \nu$.

To prove the claim,  we first show that  with probability at least   $1-2Ln^{-10}$, it holds
\begin{equation} \label{eq:maxfdl}
\max_{1\le k \le n, 1\le l \le L} \quad \abs{\vf_k^* \md_l^* \vx_0} \le  \sqrt{c_1 \log n} \norm{\vx_0}
\end{equation}
for some universal constant  $c_1>0$.
Indeed, for any fixed $1\le k \le n, 1\le l \le L$, note that 
\[
\vf_k^* \md_l^* \vx_0=\sum_{j=1}^n \bar{f}_{k,j} x_{0,j} \bar{d}_{l,j},
\]
and $|f_{k,l}|=1$. Since $\bar{d}_{l,j}, j=1,\ldots,n$ are i.i.d. centered sub-Gaussian random variables with the maximum sub-Gaussian norm $M$.
Then for any $c_1>0$, the Hoeffding's inequality gives
\begin{eqnarray*}
\PP\xkh{\abs{\vf_k^* \md_l^* \vx_0}   \ge   \sqrt{c_1 \log n} \norm{\vx_0} } & \le &  2\exp\xkh{-c \cdot \frac{ c_1 \log n \norm{\vx_0}^2}{M^2 \sum_{j=1}^n |\bar{f}_{k,j} x_{0,j}|^2} } \\
&= & 2\exp\xkh{-c \cdot \frac{c_1 \log n }{M^2}},
\end{eqnarray*}
where $c>0$ is a universal constant. Taking the constant $c_1=11M^2/c$ and taking the union bound over all $1\le k \le n, 1\le l \le L$, we obtain \eqref{eq:maxfdl}.  Under the event \eqref{eq:maxfdl}, observe that 
\begin{eqnarray} \label{eq:Ax0}
 \norm{\A(\mx_0)}^2=\sum_{l=1}^L \sum_{k=1}^n \abs{\vf_k^* \md_l^* \vx_0}^4  & \le &  c_1 \log n \norm{\vx_0}^2\cdot  \sum_{l=1}^L \sum_{k=1}^n \abs{\vf_k^* \md_l^* \vx_0}^2 \notag\\
 &= & c_1 n\log n \norm{\vx_0}^2\cdot  \sum_{l=1}^L \vx_0^* \md_l \md_l^* \vx_0,
\end{eqnarray}
where the second equality comes from the fact that $\sum_{k=1}^n \vf_k^* \vf_k=nI_n$. Finally, noting that $\vx_0^* \md_l \md_l^* \vx_0-\nu \norm{\vx_0}^2$ are i.i.d.  mean zero subexponential random variables with the maximum subexponential norm $2M^2 \norm{\vx_0}^2$, the Bernstein's inequality implies 
\begin{eqnarray*}
&& \PP\xkh{\Abs{\sum_{l=1}^L \xkh{\vx_0^* \md_l \md_l^* \vx_0- \nu \norm{\vx_0}^2}} \ge t} \\
&\le&  2\exp\xkh{-c \min\xkh{\frac{t^2}{4L M^4 \norm{\vx_0}^4},\frac{t}{2M^2 \norm{\vx_0}^2}}}.
\end{eqnarray*}
Taking $t=\nu  L \norm{\vx_0}^2$, we obtain that if $L\ge C_0 \log n$ for some constant depending only on $\mu,M$, with probability at least $1-2n^{-10}$, it holds
\begin{equation} \label{eq:x0Dl}
\sum_{l=1}^L\vx_0^* \md_l \md_l^* \vx_0 \le 2\nu L \norm{\vx_0}^2.
\end{equation}
Putting \eqref{eq:x0Dl} into \eqref{eq:Ax0}, we have
\[
 \norm{\A(\mx_0)} \le \sqrt{c_2 nL\log n} \norm{\vx_0}^2= \sqrt{c_2 nL\log n} \norm{\mx_0}.
\]
Here, $c_2=2c_1 \nu$ is a constant depending only on $M, \nu$. This gives the claim.
\end{proof}

\subsection{Proof of Theorem \ref{UpperGaussian}}

\begin{proof}[Proof of Theorem \ref{UpperGaussian}]
Since $\hat \mx$ is the optimal solution to  \eqref{model:gau} with $R=\tr(\mx_0)$ and $\mx_0$ is a feasible point, it holds 
\begin{equation} \label{eq:gaufeas}
\hat \mx \succeq \bm 0, \quad \tr(\hat\mx) \le \tr(\mx_0), \quad \mbox{and} \quad \norm{\A(\hat\mx)-\vy} \le \norm{\A(\mx_0)-\vy}.
\end{equation}
Denote $\hat\mx =\mx_0 +t \Z$ with $t>0$ and $\normf{\Z}=1$.  
Then the proof of Theorem \ref{UpperGaussian} is similar to that of Theorem \ref{mainresult}, and  the only difference is the upper bound for $\norm{\A(\Z)}$. Specifically, noting that $\vy=\A(\mx_0)+\vw$, it then follows from \eqref{eq:gaufeas} that 
\[
\norm{\A(\hat\mx)-\vy} =\norm{t\A(\Z)-\vw}\le  \norm{\A(\mx_0)-\vy}= \norm{\vw}.
\]
Therefore, it holds
\[
t \norm{\A(\Z)}^2 \le 2\nj{\A(\Z),\vw}.
\]
For the right hand side, 
\[
\nj{\A(\Z),\vw}= \sum_{l=1}^L \sum_{k=1}^n  \vf_k^* \md_l^* \Z \md \vf_k   w_{k,l}.
\]
Since $w_{k,l}$ are i.i.d. mean zero sub-Gaussian random variable with sub-Gaussian norm $\sigma$, the Hoeffding's inequality gives
\[
\PP\xkh{\Abs{\sum_{l=1}^L \sum_{k=1}^n  \vf_k^* \md_l^* \Z \md \vf_k   w_{k,l}}\ge t } \le 2\exp\xkh{-c\frac{t^2}{\sigma^2  \sum_{l=1}^L \sum_{k=1}^n  \abs{\vf_k^* \md_l^* \Z \md \vf_k}^2}}.
\]
Here, $c>0$ is a universal constant.
According to Lemma \ref{le:opnorm}, it holds
\[
\sum_{l=1}^L \sum_{k=1}^n  \abs{\vf_k^* \md_l^* \Z \md \vf_k}^2=\norm{\A(\Z)}^2 \le M^4 n^2 L.
\]
Taking $t=c_3 \omega \sigma M^2 n \sqrt{L}$ for some universal constants $\omega, c_2>0$, one has
\[
\Abs{\nj{\A(\Z),\vw}} \le c_3 \omega \sigma M^2 n \sqrt{L} 
\]
with probability at least $1- 1/6 e^{-\omega}$.  Combining with \eqref{eq:AZnorm}, we arrive at the conclusion that 
\[
\normf{\hat\mx -\mx_0} =\normf{t \Z} =t  \le C_2 \sqrt{ \frac{ \sigma \log^2 n}{\sqrt L} }.
\]
Here, $C_2>0$ is a constant depending on $\omega, M$.

\end{proof}

\subsection{Proof of Theorem \ref{th:minimax}}
In this subsection, we aim to prove the minimax lower bound stated in Theorem \ref{th:minimax}. For notational simplicity, we denote \(\va_{k,l} := \md_l \vf_k\), where \(\md_l\) and \(\vf_k\) are defined as in \eqref{eq:measu}. Furthermore, we denote by \(\mathbb{P}_{\vy|\vw}\) the likelihood of \(y_{k,l} \sim \mathcal{N}(\lvert \va_{k,l}^* \vw \rvert^2, \sigma^2)\) conditional on \(\{\vd_l\}_{l=1}^L\).
For any two probability measures \(P_1\) and \(P_2\), the Kullback-Leibler (KL) divergence between them is defined as
\[
D_{KL}(P_1, P_2) = \int \log \left( \frac{d P_1}{d P_2} \right) d P_1.
\]
The key technical tool we use to prove Theorem \ref{th:minimax} is Tsybakov’s minimax lower bound \cite{Tsybakov}, which involves constructing a finite set of hypotheses such that any two are well-separated while their corresponding distributions exhibit small mutual KL divergence, as described below.

\begin{proposition}\cite[Th. 2.7]{Tsybakov}  \label{pro:Fano}
Let \(\mathcal{P}\) be a set of distributions, and let \(X_1, \ldots, X_m\) be samples drawn from some distribution \(P \in \mathcal{P}\). Let \(\theta(P)\) be a function of \(P\) in a metric space with metric \(d\), and let \(\hat{\theta} = \hat{\theta}(X_1, \ldots, X_m)\) denote an estimator of \(\theta(P)\).
Assume that \(\{P_0, P_1, \ldots, P_N\} \subset \mathcal{P}\), where \(N \geq 3\) and \(P_0\) is absolutely continuous with respect to each \(P_j\). Suppose further that
\[
\frac{1}{N} \sum_{j=1}^N D_{KL}(P_j, P_0) \leq \frac{\log N}{16}.
\]
Then, it holds that
\[
\inf_{\hat{\theta}} \sup_{P \in \mathcal{P}} \mathbb{E}_P \left[ d(\hat{\theta}, \theta(P)) \right] \geq \frac{s}{16},
\]
where 
\[
s = \min_{0 \leq j < k \leq N} d(\theta(P_j), \theta(P_k)).
\]
\end{proposition}

The following lemma plays a key role in constructing the required hypotheses.
\begin{lemma} \label{le:seped}
Suppose that the masks $\dkh{\vd_l}_{l=1}^L$ satisfy Assumption \ref{assump:mask} with the number of masks $L \le  C_0 \log^k n$ for some fixed interger $k\ge 1$ and constant $C_0 > 0$ independent of $n$, where $n$ is sufficiently large.  For any $\vx_0 \in \R^n$, it holds with probability at least $1-2\exp(-n/40)-1/\log n$, there exists a collection $\mathcal M$ of $N=\exp(n/30)$ distinct vectors obeying the following properties:
\begin{itemize}
\item[\rm{(i)}] $\vx_0 \in \mathcal M$;
\item[\rm (ii)] for all $\vw_k, \vw_l \in \mathcal M$, 
\[
 \frac{\sigma}{160\sqrt{6}M \norm{\vx_0}}  \cdot \sqrt{\frac{1}{L \log^2 n}} \le \norm{\vw_k-\vw_l} \le \frac{\sigma}{40\sqrt{6}M \norm{\vx_0}} \cdot \sqrt{\frac{1}{L \log^2 n}};
\]
\item[\rm (iii)] for all $\vw\in \mathcal M$, 
\[
 \max_{ 1\le k\le n, 1\le l\le L } \abs{\va_{k,l}^* (\vw-\vx_0)} \le \frac{\sigma}{80\sqrt{3}}\cdot \sqrt{\frac{1}{L\log n \norm{\vx_0}^2}}.
\]
\end{itemize}
Here, $M$ is  defined in \eqref{assump:mask}.
\end{lemma}
\begin{proof}
We first construct a set $\mathcal M_1$ of exponentially many vectors centered around $\vx_0$ such that properties (i) and (ii) hold. Then, we verify that a subset $\mathcal M \subset \mathcal M_1$, which also contains exponentially many vectors, satisfies  property (iii).  To this end,  define a random vector 
\[
\vw:=\vx_0+\frac{\sigma}{80\sqrt{6}M}\cdot \sqrt{\frac{1}{nL \log^2n \norm{\vx_0}^2}} \cdot \vz,
\]
where $\vz \sim \mathcal N(0,I_n)$ and $M$ is the parameter given in \eqref{assump:mask}.
 The collection $\mathcal M_1$ is constructed by generating $N=\exp(n/20)$ independent copies $\vw_k$ of $\vw$. Notice that for any $\vw_k, \vw_l\in \mathcal M_1$, 
\[
\vw_k-\vw_l=\frac{\sigma}{80\sqrt{6}M}\cdot  \sqrt{\frac{1}{nL \log^2n \norm{\vx_0}^2}}  \xkh{\vz_k-\vz_l},
\]
where  $\vz_k$ and $ \vz_k$ are independent standard Gaussian random vectors.  By \cite[Theorem 3.1.1]{Vershynin2018}, we have
\[
\PP\dkh{0.5 \sqrt{n} \le \norm{\xkh{\vz_k-\vz_l}} \le 2 \sqrt{n} } \ge 1-2\exp(-n/8).
\]
Applying the union bound over all pairs $\vw_k, \vw_l$, it follows that 
\[
\frac{\sigma}{160\sqrt{6}M} \cdot \sqrt{\frac{1}{L\log^2n \norm{\vx_0}^2}}  \norm{\vw_k-\vw_l} \le \frac{\sigma}{40\sqrt{6}M}\cdot  \sqrt{\frac{1}{L\log^2n \norm{\vx_0}^2}}, \quad \forall k\neq l
\]
holds with probability at least 
\[
1-2\exp(-n/8) \cdot  N^2 \ge 1-2\exp(-n/40).
\]
This gives property (ii).

For property (iii),  note  that when $\vz \sim  \mathcal N(0,I_n)$, conditional on $\dkh{\va_{k,l}}$, the standard Gaussian concentration inequality implies
\[
\PP\dkh{\abs{\va_{k,l}^* \vz} \ge \sqrt{2\log n} \norm{\va_{k,l}}} \le \frac1{n^2}.
\]
Taking the union bound for all $1\le k\le n, 1\le l\le L$, we have
\[
\PP\dkh{ \max_{ 1\le k\le n, 1\le l\le L }\abs{\va_{k,l}^* \vz}  \le \sqrt{2\log n} \norm{\va_{k,l}}} \ge 1- \frac{L}{n}.
\]
Since $\norm{\va_{k,l}} \le M \sqrt{n}$,  it follows that  for any $\vw\in \mathcal M_1$,  with probability at least $1- \frac{L}{n}$, 
\begin{equation} \label{eq:eaw}
 \max_{ 1\le k\le n, 1\le l\le L } \abs{\va_{k,l}^* (\vw-\vx_0)} \le \frac{\sigma}{80\sqrt{3}}\cdot   \sqrt{\frac{1}{L \log n \norm{\vx_0}^2}}.
\end{equation}
Define the subset
\[
\mathcal M=\dkh{\vw \in \mathcal M_1 :  \vw\;  \text{ satisfies the inequality  \eqref{eq:eaw} }}.
\]
Next, we show that $\mathcal M$ still contains exponentially many vectors. For each $\vw_k \in \mathcal M_1$, define a Bernoulli random variable $\xi_k$ such that  $\PP(\xi_k=0)$ if \eqref{eq:eaw} holds for $\vw_k$,  and $\PP(\xi_k=1)$ otherwise.  
Note that $\E \xi_k=L/n$. The Markov's inequality gives 
\[
\PP\dkh{\sum_{k=1}^N \xi_k \ge \frac{NL \log n}{n} }\le \frac{1}{\log n}.
\]
Therefore, with probability $1-1/\log n$, there exist at least 
\[
N-\frac{NL \log n}{n}=N\xkh{1-\frac{L}{n}} = \exp\xkh{\frac{n}{20}}\xkh{1-\frac{\log^k n}{n}} \ge \exp\xkh{\frac{n}{30}}
\]
vectors in $\mathcal M_1$ satisfies the inequality \eqref{eq:eaw}, establishing condition (iii).
\end{proof}

Now we are ready to prove the minimax optimality.

\begin{proof}[Proof of Theorem \ref{th:minimax}]
According to Lemma \ref{le:seped}, for any $\vx_0 \in \C^n$ with $\norm{\vx_0}\ge \sigma$, with probability at least $1-2\exp(-n/40)-1/\log n$, there exists a collection $\mathcal M$ of $N=\exp(n/30)$ distinct vectors obeying conditions (i), (ii), and (iii).  For each $\vw_k \in \mathcal M \setminus \dkh{\vx_0}$, let $P_k$ denote the distribution of $\PC{\vy}{\vw_k}$, and let $P_0$ denote the distribution of $\PC{\vy}{\vx_0}$. In the Gaussian model where the noises $w_{k,l}\sim \mathcal N(0,\sigma^2)$, it follows that 
\begin{equation} \label{eq:conden}
\PC{\vy}{\vw_k} \sim \mathcal N(\abs{A\vw_k}^2, \sigma^2 I_m),
\end{equation}
where $A \in \C^{m\times n}$ is a matrix whose rows are $\va_{k,l}$,  and  $\abs{\cdot}$ acts entrywise. Thus,  the density function satisfies
\[
\PC{\vy=\vx}{\vw_k}=\prod_{1\le k\le n, 1\le l\le L} \frac{1}{\sqrt{2\pi} \sigma} \exp\xkh{-\frac{\xkh{x_{k,l}- \abs{\va_{k,l}^* \vw_k}^2}^2}{2\sigma^2}},
\]
where $\vx=\dkh{x_{k,l}}$.
Condition on $\dkh{\va_{k,l}}$, the KL divergence between $\PC{\vy}{\vw_k}$ and $\PC{\vy}{\vw_l}$ obeys 
\begin{eqnarray*}
&& D_{KL}(P_k,P_0) = \int \log\xkh{\frac{d P_k}{d P_0}} d P_k \\
&=& \E_{P_k} \zkh{ \log \frac{\prod_{k,l } \frac{1}{\sqrt{2\pi} \sigma} \exp\xkh{-\frac{\xkh{x_{k,l}- \abs{\va_{k,l}^* \vw_k}^2}^2}{2\sigma^2}}}{\prod_{k,l } \frac{1}{\sqrt{2\pi} \sigma} \exp\xkh{-\frac{\xkh{x_{k,l}- \abs{\va_{k,l}^* \vx_0}^2}^2}{2\sigma^2}}}} \\
&=& \frac{1}{2\sigma^2} \sum_{1\le k\le n, 1\le l\le L} \Abs{\abs{\va_{k,l}^* \vw_k}^2-\abs{\va_{k,l}^* \vx_0}^2 }^2 \\
& \le & \frac{1}{2\sigma^2} \sum_{1\le k\le n, 1\le l\le L}  \Abs{\va_{k,l}^* \xkh{\vw_k-\vx_0}}^2 \xkh{2  \Abs{\va_{k,l}^* \vx_0}+ \Abs{\va_{k,l}^* \xkh{\vw_k-\vx_0}}}^2 \\
& \le & \frac{1}{2\sigma^2} \cdot (nL) \cdot  \frac{\sigma^2}{1820 L \log n \norm{\vx_0}^2} \cdot \xkh{2 \sqrt{6 \log n} \norm{\vx_0}+ \frac{\sigma}{40\sqrt{3}}\cdot \sqrt{\frac{1}{L\log n \norm{\vx_0}^2}} }^2 \\
& \le & \frac{n}{480},
\end{eqnarray*}
where the third line follows from the Gaussian random distribution \eqref{eq:conden}, the forth line comes from the triangle inequality, and the fifth line arises from property (iii) of Lemma \ref{le:seped}, inequality \eqref{eq:maxfdl} (with the constant $6$ taken  for simplicity),  and the fact that $\norm{\vx_0} \ge \sigma$.  Consequently, 
\[
\frac{1}{N-1} \sum_{\vw_k \in \mathcal M \setminus \dkh{\vx_0}} D_{KL}(P_k,P_0)  \le  \frac{n}{480} \le \frac{\log(N-1)}{16}.
\]
By applying Proposition \ref{pro:Fano}, Tsybakov minimax lower bound yields 
\[
\inf_{\hat\vx} \sup_{\vx_0 \in \mathcal M} \E\zkh{\norm{\hat\vx-\vx_0} | \dkh{\va_{k,l}}} \ge \frac{c_0 \sigma}{ \norm{\vx_0}} \cdot \sqrt{\frac1{L\log^2n}}.
\]
Here, $c_0$ is a constant depending only on $M$. Finally, since all vectors $\vw_k \in \mathcal M$ are clustered around $\vx_0$,  any reasonable estimator satisfies $\mbox{dist}(\hat\vx,\vx_0)=\norm{\hat\vx-\vx_0}$. This completes the proof. 
\end{proof}

\section{Numerical experiments} \label{sec:4}
In this section, we present numerical experiments to verify that the recovery bound $\norm{\vw}/\sqrt m$ for phase retrieval from CDP  is rate-optimal.  In our experiments, the target vector $\vx_0 \in \C^n$ is drawn randomly from the standard complex Gaussian distribution, i.e., $\vx_0 \sim \mathcal{N}(0, I_n) + i \mathcal{N}(0, I_n)$.  The noise vector $\vw \in \R^{nL}$ is a real Gaussian random vector with entries $w_{k,l} \sim \mathcal{N}(0, 1)$. We perform simulations using octanary masks as \eqref{eq:octanary}, as well as ternary masks \cite{CDPLi}, where the mask entries are distributed as 
\[
d \sim \left\{
\begin{array}{rl}
1 & \text{with prob. } 1/4 \\
0 & \text{with prob. } 1/2 \\
-1 & \text{with prob. } 1/4
\end{array}
\right. .
\]
%
We fix $n = 128$ and vary the number of masks $L$ from 5 to 50. For each fixed $L$, we run 100 trials and compute the average ratio $\rho_m$ defined as
\[
\rho_m = \frac{\normf{\hat{\mx} - \mx_0}}{\norm{\vw} / \sqrt{m}}.
\]
The PhaseLift program is implemented using CVX toolbox.
Figure \ref{figure:1} depicts the values of $\rho_m$ against the number of masks $L$.  It can be observed that $\rho_m$ stabilizes around approximately $0.23$ and $0.45$ for octanary masks and ternary masks, respectively.
This demonstrate that the recovery bound $\norm{\vw}/\sqrt m$ is rate-optimal.

\begin{figure}
\centering
\includegraphics[width=0.75\textwidth]{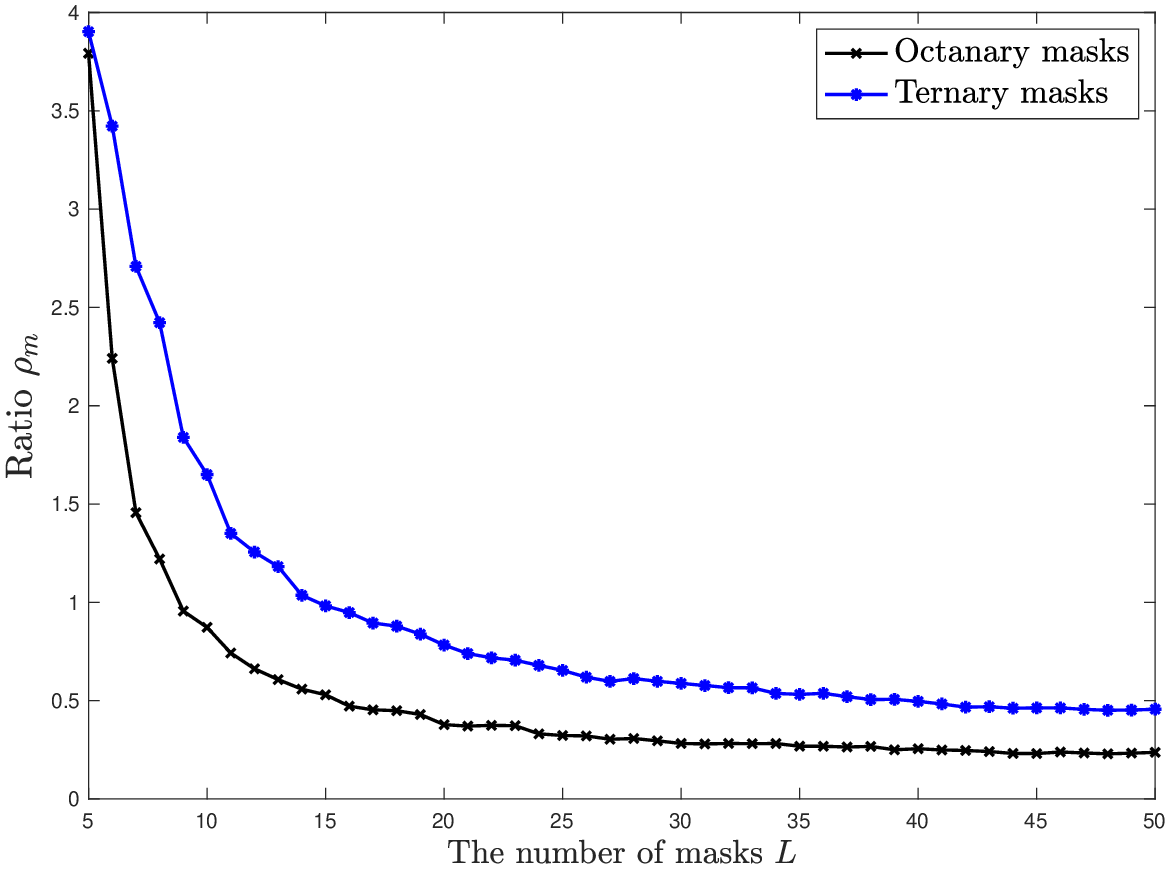}
\caption{The ratio $\rho_m$ versus the number of masks $L$ under adversarial noises with $n=128$.}
\label{figure:1}
\end{figure}

\section{Discussions} \label{sec:5}
This paper investigated the recovery bounds of PhaseLift programs for phase retrieval from coded diffraction patterns. Nearly sharp error bounds were established for both adversarial noise and zero-mean Gaussian noise cases in the large noise levels, utilizing descent cone analysis and exact dual certificates. These results partially answer the conjecture posed in \cite[p. 173]{Mahdithesis} up to a logarithmic factor.

Several interesting directions for future research remain. First, our error bounds in the small noise levels are sub-optimal. Refining these bounds to eliminate the  gap would be valuable. 
Second, phase retrieval under sparsity assumptions has garnered significant attention \cite{li2013sparse}, so establishing recovery bounds for PhaseLift in this context presents a promising avenue for further investigation.
Finally, several non-convex optimization algorithms have recently been proposed to solve phase retrieval from coded diffraction patterns, which are more efficiently than PhaseLift  for large-scale problems. Developing rigorous theoretical recovery guarantees for these estimators is an important direction for future work.

\appendix
\section{Proof of auxiliary lemmas} \label{sec:6}
\begin{proof}[Proof of Lemma \ref{le:appdual}]
The explicit construction of the approximate dual certificate is given in \cite[Proposition 18]{CDPGross}, and \eqref{eq:appdual1} is the direct consequences. To finish the proof, we only need to check \eqref{eq:appdual2}.  Specifically, the construction of $\Y$ 
given in \cite[Proposition 18]{CDPGross} is based on golfing scheme with expression as
\[
 \Y=\sum_{j=1}^{r+2} \xkh{ \mathcal R(\mQ_{j-1}) -\tr(\mQ_{j-1})I_n }.
\]
Here, 
\[
r=\lceil \frac12 \log_2 n \rceil +\lceil \log_2(M^2/\nu) \rceil +1,
\]
and the matrices $\mQ_j$ are defined iteratively as
\[
 \mQ_j=\mathcal P_T \xkh{\mQ_{j-1} + \tr(\mQ_{j-1}) \mQ_{j-1} - \mathcal R(\mQ_{j-1}) }, \quad \mQ_{0}=\mx_0,
\]
and $\mathcal R(\cdot)$ is an operate defied as
\[
\mathcal R(\mQ)=\frac{1}{\nu^2 n L} \sum_{l=1}^L \sum_{k=1}^n  \nj{ \md_l\vf_k \vf_k^* \md_l^* , \mQ } \md_l\vf_k \vf_k^* \md_l^* \1_{\abs{\nj{ \md_l\vf_k \vf_k^* \md_l^* , \mQ}} \le 4M^2 \gamma \log n \normf{\mQ}}
\]
with $\gamma \ge 1$ being a fixed constant.  Furthermore, as shown in \cite[Proposition 18]{CDPGross}, one has 
\begin{equation} \label{eq:normQ}
\norm{\mQ_0}=1,  \norm{\mQ_1} \le \frac{1}{\sqrt{2\log n}},  \norm{\mQ_2} \le \frac{1}{2\log n},  \norm{\mQ_j} \le  \frac{1}{\log n} 2^{-(j-1)}
\end{equation}
Based on the above construction and the expression of $\A^*$, each entry of the dual certificate $\lambda' \in \R^{nL}$ is
\[
\lambda'_{k,l} =\frac{1}{\nu^2 n L}  \sum_{j=1}^{r+2} \xkh{ \nj{ \md_l\vf_k \vf_k^* \md_l^* , \mQ_{j-1} }   \1_{ \abs{\nj{ \md_l\vf_k \vf_k^* \md_l^* , \mQ_{j-1}}} \le 4M^2 \gamma \log n \normf{\mQ_{j-1}}} - \tr(\mQ_{j-1}) }.
\]
Observing that each $\mQ_j$ is a Hermitian matrix and lies in the tangent space $T$ with $\rank(\mQ_j) \le 2$, simple calculation gives that 
\begin{eqnarray*}
\norms{\lambda}_{\infty}  & \le & \frac{1}{\nu^2 n L}  \sum_{j=1}^{r+2}  \xkh{4M^2 \gamma \log n \normf{\mQ_{j-1}} +  \norms{\mQ_{j-1}}_*  } \\
&\le &  \frac{4M^2 \gamma \log n + \sqrt2 }{\nu^2 n L}  \sum_{j=1}^{r+2}  \normf{\mQ_{j-1}} \\
&\le &  \frac{4M^2 \gamma \log n + \sqrt2 }{\nu^2 n L} \xkh{1+ \frac{1}{\sqrt{2\log n}} + \sum_{j=3}^{r+2}  \frac{1}{\log n} 2^{-(j-1)} } \\
&\le & \frac{c' \log n}{ nL},
\end{eqnarray*}
where the third inequality comes from \eqref{eq:normQ}.
This immediately gives
\[
\norm{\lambda} \le \frac{c' \log n}{ \sqrt{nL}}.
\]
\end{proof}

\bibliographystyle{plain}

\end{document}